\documentclass[a4paper,11pt]{amsart}
\usepackage{amsfonts}
\usepackage{amssymb}
\usepackage{amsmath}
\usepackage{newlfont}
\usepackage[mathcal]{euscript}
\usepackage{mathrsfs}
\usepackage{txfonts}
\usepackage{graphicx}
\usepackage{color}
\usepackage{url}
\usepackage{enumerate}

\newtheorem{theorem}{Theorem}

\newtheorem{hthm}{Theorem?}
\newtheorem{lemma}[theorem]{Lemma}

\newtheorem{corollary}[theorem]{Corollary}

\theoremstyle{definition}
\newtheorem{definition}{Definition}

\newcommand{\gluen}[1]{\stackrel{{\scriptscriptstyle (#1)}}{{\scriptstyle \ast}}}

\DeclareMathOperator{\diam}{diam}
\DeclareMathOperator{\lang}{\mathcal{L}}
\DeclareMathOperator{\dind}{\delta}
\DeclareMathOperator{\adsymbol}{\Delta}
\newcommand{\adsymb}[1]{\adsymbol_{\mathbf{#1}}}
\newcommand{\adone}{\adsymb{1}}

\DeclareMathOperator{\ones}{\mathbf{1}}

\DeclareMathOperator{\ind}{\mathbb{I}}
\newcommand{\SP}{\Sigma_P}
\newcommand{\spacing}[1]{X(#1)}

\newcommand{\spp}{\spacing{P}}

\DeclareMathOperator{\Diff}{Diff}
\DeclareMathOperator{\Equal}{Equal}

\DeclareMathOperator{\ld}{\underline{d}}
\DeclareMathOperator{\ud}{\overline{d}}
\DeclareMathOperator{\ad}{{d}}

\DeclareMathOperator{\ubd}{{BD}^*}

\newcommand{\eps}{\varepsilon}

\begin{document}
\title[Entropy and chaos in hereditary shifts]{Topological entropy and distributional chaos in hereditary shifts with applications to spacing shifts and beta shifts}
\author{Dominik Kwietniak}
\address{Institute of Mathematics, Jagiellonian University in Krak\'{o}w, ul. {\L}ojasiewicza 6, 30-348 Krak\'{o}w, Poland}
   \email{dominik.kwietniak@uj.edu.pl}



   \dedicatory{Dedicated to the memory of Professor Andrzej Pelczar (1937-2010).}

   \date{\today}
\maketitle
\begin{abstract}
Positive topological entropy and distributional chaos are characterized for hereditary shifts.
A hereditary shift has positive topological entropy if and only if it is DC$2$-chaotic
(or equivalently, DC$3$-chaotic) if and only if it is not uniquely ergodic.
A hereditary shift is DC$1$-chaotic if and only if it is not proximal (has more than one minimal set).
As every spacing shift and every beta shift is hereditary the results apply to those classes of shifts.
Two open problems on topological entropy and distributional chaos
of spacing shifts from an article of Banks et al. are solved thanks to this characterization.
Moreover, it is shown that a spacing shift $\Omega_P$ has positive topological entropy if and only if
$\mathbb{N}\setminus P$ is a set of
Poincar\'{e} recurrence. Using a result of K\v{r}\'{\i}\v{z} an example of a proximal spacing shift
with positive entropy is constructed.
Connections between spacing shifts and difference sets are revealed and the methods of this paper are used to obtain new proofs of some results on
difference sets.
\end{abstract}

\section{Introduction}
A \emph{hereditary shift} is a (one-sided) subshift $X$ such that $x\in X$ and $y\le x$
(coordinate-\-wise) imply $y\in X$. As far as we known, hereditary shifts were introduced
by Kerr and Li in \cite[p. 882]{KerrLi}). We are not aware of any further research on hereditary shifts.
The notion of hereditary shift generalizes at least
two classes of subshifts whose importance has been established in the literature: spacing shifts and beta shifts.

Given $\beta>1$ the (one-sided) \emph{beta shift} $\Omega_{\beta}$ is a subset of $\Omega_{\lceil\beta\rceil}=\{0, 1 ,\ldots, \lfloor\beta\rfloor\}^\mathbb{N}$ defined as the closure (with respect to the product topology) of
the set of sequences arising as a $\beta$-expansion of numbers from $[0,1]$. Beta shifts were first considered
by R\'{e}nyi \cite{R} and are a family of symbolic spaces with an extremely rich structure and
a profound connection to number theory, tilings, and the dynamics
of systems with discontinuities.

By a \emph{spacing shift} $\Omega_P$,
where the parameter $P$ is a subset of the positive integers $\mathbb{N}$, we mean
the set of all infinite binary sequences for which the occurrences of $1$'s
have distances lying in $P$.
In other words, $\Omega_P$ contains only those sequences $\omega=(\omega_i)$
that $\omega_i=\omega_j=1$ and $i\neq j$ imply $|i-j|\in P$.
Spacing shifts were introduced by Lau and Zame in \cite{LZ}
(see also \cite[pp. 241-2]{deVries}).
Spacing shifts served for Lau and Zame as counterexamples. It seems that spacing shifts
were hardly explored afterwards, except in \cite{B,SVJL,KO-WM} where again they were used to construct
counterexamples. Recently, a more thorough study of spacing shifts was conducted in \cite{spacing}.
It was revealed that spacing shifts exhibit wide variety of
interesting dynamics worth to be exploited further.

Our work extends and completes the line of investigation of \cite{spacing} to a broader class of hereditary shifts, which also contains all beta shifts.
In particular, we solve two open problems (Questions 4 and 5 of \cite{spacing}), regarding topological entropy
and distributional chaos in the more general context of hereditary shifts.

In order to classify hereditary shifts, notice first that the fixed point $0^\infty$ belongs to any hereditary shift, hence the atomic measure $\mu_0$ carried by
this fixed point is an invariant measure of the system. Therefore one can divide all hereditary shifts into two major classes:
\renewcommand{\theenumi}{\Roman{enumi}}
\renewcommand{\labelenumi}{\theenumi.}
\begin{enumerate}
  \item those with a unique invariant measure $\mu_0$ (\emph{uniquely ergodic} ones), and \label{class:I}
  \item those which have another invariant measure. \label{class:II}
\end{enumerate}
Our main result (contained in Theorems \ref{thm:hereditary-entropy-density}, \ref{thm:ergodic-equivalences}, and \ref{thm:dc2}) states that for hereditary shifts the above classification coincides with at least three other natural classifications: zero versus positive topological entropy, lack of any DC$2$, or even DC$3$ distributionally scrambled pair versus presence of uncountable set of distributionally scrambled pairs, that is, distributional chaos DC$2$ (or equivalently, DC$3$-chaos), and zero Banach density of occurrences of symbol $1$ in all points of $X$ versus existence of a point in $X$ with 1's appearing with positive upper Banach density. 

Recall here, that distributional chaos was introduced
in the setting of maps of the interval, as an equivalent condition for positive topological
entropy (see \cite{SS}). Although this equivalence does not hold in general,
distributionally chaotic dynamics is a source of interesting research problems (see \cite{D,O,O2,P}).

Another classification of hereditary shifts is this
\renewcommand{\theenumi}{\Alph{enumi}}
\renewcommand{\labelenumi}{\theenumi.}
\begin{enumerate}
  \item those with a unique minimal set, consisting of a single fixed point $0^\infty$ (\emph{proximal} ones), and \label{class:A}
  \item those which have another minimal set. \label{class:B}
\end{enumerate}
\renewcommand{\theenumi}{\arabic{enumi}}
\renewcommand{\labelenumi}{\theenumi.}

Notice that any shift in class \eqref{class:I} must be in class \eqref{class:A}, as any minimal set carries at least one invariant measure. In other words, the class (\ref{class:I}\ref{class:B}) is empty. In Theorem \ref{thm:dc1} we characterize hereditary shifts exhibiting distributional chaos of type $1$ (DC$1$-chaos) as
non-proximal shifts, that is, those in class \eqref{class:B}. It is known that every beta shift is in class
(\ref{class:II}\ref{class:B}). It follows that every beta shift is DC$1$-chaotic.
Next, we use our characterization of hereditary shifts with positive entropy as those presenting distributional chaos of type $2$ (DC$2$-chaotic ones), and the example constructed by K\v{r}\'{\i}\v{z} \cite{Kriz} (and refined in \cite{McC} according to the idea of Ruzsa), to show in Theorem \ref{thm:kriz} the existence of a topologically weakly mixing spacing shift with unique minimal set $0^\infty$ but not unique invariant measure, hence proving there exists a DC$2$-chaotic spacing shift, which is not DC$1$-chaotic (there exists a hereditary shift of class (\ref{class:II}\ref{class:A})). This answers \cite[Question 4]{spacing}.
Finally, Theorem \ref{thm:mix} proves that the class (\ref{class:I}\ref{class:A}) is also non-empty and there are non-spacing and non-beta hereditary shifts.



Further, we prove in Theorem \ref{thm:Poincare} that the entropy of a spacing shift $\Omega_P$ is positive
if and only if $\mathbb{N}\setminus P$
is not a set of recurrence, or, equivalently, $P$ intersects nontrivially
any set of recurrence. Here, following Furstenberg (see \cite[p. 219]{F2}), we say that
$R\subset\mathbb{N}$ is a \emph{set of recurrence}
if for every measure preserving system $(X,\mathcal{X},\mu,T)$
and any set $A\in\mathcal{X}$ with $\mu(A)>0$ %
there is an $r\in R$ such that $\mu(T^{-r}(A)\cap A)>0$.
The later result links the topological entropy of
spacing shifts with the return times sets appearing in a generalization of Poincar\'{e} recurrence theorem.
At first sight this
connection is quite unexpected, since it ties a measure theoretic notion of Poincar\'{e}
recurrence with the notion of topological entropy of some subshift, which in turn
may be expressed in combinatorial terms only. Unfortunately, the problem of intrinsic characterization of sets of recurrence
is notoriously elusive, and our result turns out to be only its restatement.
But we still believe that our approach opens the possibility to explore sets of
recurrence from the new a perspective.

Finally, we would like to point out a
connection of spacing shifts with combinatorial number theory.
It is possible to apply the results on spacing shifts to explore \emph{difference sets},
that is, sets of the form $A-A=\{k-l:k,l\in A,\, k>l\}$,
where $A\subset\mathbb{N}$.
Identifying, as above, infinite binary sequences with characteristic functions
of subsets of $\mathbb{N}$ one observes that for any $P$ the spacing
shift $\Omega_P$ contains the sequences representing
such sets $A\subset\mathbb{N}$ that $A-A\subset P$.
Therefore it is natural to ask how the properties of a difference set $P=A-A$
are related to the spacing shift
$\Omega_P$. In this direction our work provides a topological version of the
Furstenberg ergodic proof that for any set $A$ with
positive upper Banach density the set $A-A$ 
contains the difference set of some set $D$ with positive
asymptotic density (see the proof of Theorem \ref{thm:F} below and \cite[Corollary to thm. 3.20]{F}).

\subsection*{Acknowledgements} Results contained in the present paper were presented by the author at the
Visegrad Conference on Dynamical Systems, held in Bansk\'{a} Bystrica between 27 June and 3 July 2011, and
at the 26th Summer Conference on Topology and Its Applications
hosted in July 26-29, 2011 by The City College of CUNY. Note that \cite[Question 5]{spacing} was also independently solved
by Dawoud Ahmadi Dastjerdi and Maliheh Dabbaghian Amiri in \cite{iran}. The authors of \cite{iran} also
proved that for a spacing shift zero entropy implies proximality. This is also a corollary of 
the more general Theorem~\ref{thm:ergodic-equivalences} presented below. 
The author is greatly indebted to professor  
Mike Boyle, Jian Li, and Piotr Oprocha for several helpful comments concerning the subject of this paper. 
The anonymous referee of the previous version of this paper provided a superb report with many useful suggestions, 
which are included in the present form. The research leading to this paper were supported by the grant IP2011
028771.


\section{Basic notions and conventions}
A \emph{dynamical system} is a pair $(X,f)$, where $X$ is a compact metric space,
and $f\colon X\mapsto X$ is a continuous map. We usually denote the metric on $X$ by $d$.
By an \emph{invariant set} we mean any
set $K\subset X$ such that $f(K)\subset K$. Any nonempty, closed and invariant set
$K$ is identified with the \emph{subsystem} $(K,f|_K)$ of $(X,f)$.
A dynamical system is \emph{minimal} if it
has no proper subsystems. A point $x\in X$ is called a \emph{minimal} if it belongs to some
minimal subsystem. A pair $(x,y)\in X\times X$ is a \emph{proximal pair} if
\[
\liminf_{n\to\infty} d(f^n(x),f^n(y))=0.
\]
We say that a dynamical system $(X,f)$ is \emph{proximal} if every pair in $X\times X$
is a proximal pair.

By a \emph{Lebesgue space} we mean a triple $(X,\mathcal{X},\mu)$, where $X$ is a
Polish space, $\mathcal{X}$ is the $\sigma$-algebra of Borel sets on $X$, and $\mu$
is a probability measure on $\mathcal{X}$. We ignore null sets, and accordingly
we will assume that that all probability spaces are complete.
A \emph{measure preserving system} is
a quadruple $(X,\mathcal{X},\mu,T)$, where $(X,\mathcal{X},\mu)$ is a Lebesgue
space, and $T\colon X\mapsto X$ is a measurable map preserving $\mu$, that is,
$T^{-1}(B)\in\mathcal{X}$ and $\mu(T^{-1}(B))=\mu(B)$ for every $B\in\mathcal{X}$.
If $(X,f)$ is a dynamical system, then there always exists an \emph{invariant}
measure, that is, a complete Borel probability measure $\mu$, such that
$(X,\mathcal{X},\mu,f)$ is a measure preserving system. An invariant measure
for $(X,f)$ is
\emph{ergodic} if the only members $B$ of $\mathcal{X}$ with $f^{-1}(B)=B$ satisfy $\mu(B)=0$ or $\mu(B)=1$.
A dynamical system $(X,f)$
is \emph{uniquely ergodic} if it has exactly one invariant measure.







Given an infinite set of positive integers $S$ we enumerate $S$ as an increasing sequence $s_1<s_2<\ldots$ and define \emph{the sum set} $\text{FS}(S)$ of $S$ by
\[
  \text{FS}(S)=\{s_{n(1)}+\ldots+s_{n(k)}:n(1)<\ldots<n(k),\,k\in\mathbb{N}\}.
\]

We say that a set $A\subset \mathbb{N}$ is
\begin{enumerate}
    \item \emph{thick}, if it contains arbitrarily long blocks of consecutive integers, that is,
        for every $n>0$ there is $k\in\mathbb{N}$ such that $\{k,k+1,\ldots,k+n-1\}\subset A$,
    \item \emph{syndetic}, if it has bounded gaps, that is, for some $n>0$ and every
        $k\in\mathbb{N}$ we have $\{k,k+1,\ldots,k+n-1\}\cap A\neq \emptyset$,
    \item an \emph{$\text{IP}$-set} if it contains the sum set $\text{FS}(S)$ of some infinite set $S\subset\mathbb{N}$.
    \item \emph{$\Delta$-set} if it contains the difference set $A-A$ of some infinite set $A\subset\mathbb{N}$,
    \item \emph{piecewise syndetic} if it is an intersection of a thick set with a syndetic set,
    \item \emph{$\Delta^*$-set} (\emph{$\text{IP}^*$-set}), if it has non-empty intersection with every $\Delta$-set
        ($\text{IP}$-set, respectively).
\end{enumerate}
Note that for some authors $\text{IP}$-sets are exactly the finite sum sets as defined above (see, e.g., Furstenberg's book \cite{F}).

By the \emph{upper density} of a set $A\subset\mathbb{N}$ we mean the number
\[
\ud(A) = \limsup_{n\to\infty} \frac{\#A\cap\{1,\ldots,n\}}{n}.
\]
If limes superior above is actually the limit, then we write $\ad(A)$ instead of $\ud(A)$, and call it the \emph{asymptotic density} of $A$.
The \emph{upper Banach density} of a set $A\subset\mathbb{N}$ is the number
\[
\ubd(A)=\limsup_{n-m\to\infty} \frac{\#A\cap\{m,m+1,\ldots,n-1\}}{n-m}.
\]
Given a dynamical system $(X,f)$ and sets $A,B\subset X$ we define
the \emph{set of transition times from  $A$ to $B$} by
\[
N(A,B)=\{n>0: f^n(A)\cap B \neq\emptyset\}.
\]
If $x\in X$, then $N(x,B)=\{n>0: f^n(x)\in B\}$ denotes the
\emph{set of visiting times}. There are no commonly accepted names
for the sets $N(A,B)$ and $N(x,B)$. Some authors (see, e.g., \cite{LiJian}) prefer to call them
the set of \emph{hitting times of $A$ and $B$}, and the set of times
\emph{$x$ enters into $B$}, respectively. Note that $N(x,B)=N(\{x\},B)$.
Many recurrence properties of a dynamical system $(X,f)$ may
be  characterized in terms of sets of transition (visiting) times
sets. For the purposes of the present paper we will state
these equivalent characterizations in the theorems below and omit
the standard definitions.

\begin{theorem}\label{thm:transitivity}
A dynamical system $(X,f)$ is
\begin{enumerate}
    \item   \emph{mixing} if and only if $N(U,V)$ is cofinite for any pair of nonempty open sets $U,V\subset X$,
    \item   \emph{weakly mixing} if and only if $N(U,V)$ is thick for any pair of nonempty open sets $U,V\subset X$,
\end{enumerate}
\end{theorem}

The first equivalence above is straightforward, the second one follows, e.g., from \cite[Proposition II.3]{Fdis}.
For the proof of the next theorem see, e.g., \cite{BF}, and consult \cite[Section 5]{LiJian2} for more information.

\begin{theorem}\label{thm:minimality}
Let $(X,f)$ be a dynamical system. A point $x\in X$ is minimal
if and only if for every open neighborhood $U$ of $x$ the set
$N(x,U)$ is syndetic. Moreover, a nonempty open set $U\subset X$
contains a minimal point if and only if $N(x,U)$ is piecewise
syndetic for some $x\in X$.
\end{theorem}


\section{Spacing shifts}

Let $n\ge 2$ and $\Lambda_n=\{0,1,\ldots,n-1\}$ be equipped with the discrete topology.
We endow the space of all infinite sequences of symbols from $\Lambda_n$
indexed by the positive integers $\mathbb{N}$ with the product
topology, and denote it by $\Omega_n=\Lambda_n^\mathbb{N}$. The reader should remember
(especially reading section~\ref{sec:dc}) that we will equip $\Omega_n$ with a compatible metric $\rho$ given by
\[
\rho(\omega,\gamma)=\left\{
                      \begin{array}{ll}
                        n^{-\min\{k\in\mathbb{N}:\omega_k\neq\gamma_k\}}, &  \hbox{if $\omega\neq\gamma$;} \\
                        0, & \hbox{if $\omega  = \gamma$.}
                      \end{array}
                    \right.
\]
The \emph{shift} transformation $\sigma$ acts on $\omega\in\Omega_n$ by shifting it one position to the
left. That is,
$\sigma\colon\Omega_n\mapsto\Omega_n$ given by $(\sigma(\omega))_i=\omega_{i+1}$, where $\omega=(\omega_i)$.
A \emph{subshift} is any nonempty closed subset $X$ of $\Omega_n$ such that $\sigma(X)\subset X$.
If $n=2$, then we call $X\subset\Omega_2$ a binary subshift.

A \emph{word} of length $k$ (a \emph{$k$-word} for short) is a sequence $w=w_1w_2\ldots w_k$ of elements
of $\Lambda_n$. The length of a word $w$ is denoted as $|w|$.
We will say that a word $u=u_1u_2\ldots u_k$ \emph{appears} in a word $w=w_1w_2\ldots w_n$
at position $t$, where $1\le t \le n-k+1$ if  $w_{t+j-1}=u_j$ for $j=1,\ldots,k$.
Similarly, a word $u$ appears in $\omega=(\omega_i)\in \Omega$ at position $t\in \mathbb{N}$ if
$\omega_{t+j-1}=u_j$ for $j=1,\ldots,k$. A \emph{cylinder} given by a word $w$ is the set $[w]$ of all sequences $\omega\in\Omega_n$ such that
$w$ appears at position $1$ in $\omega$. The collection of all cylinders form a base for the topology on $\Omega_n$.

The \emph{concatenation} of words $w$ and $v$ is a
sequence $u=wv$ given by $u_i=w_i$ for $1\le i \le |w|$ and $u_i=v_{i-|w|}$ for $|w|+1\le i \le |w|+|v|$.
If $u$ is a word, and $n\ge 1$, then $u^n$ is the
concatenation of $n$ copies of $u$. 
Then $u^\infty$ has its obvious meaning.

If $S\subset \Omega_n$, then the \emph{language} of $S$
is the set $\mathcal{L}(S)$ of all nonempty words which appear at some position in some $x\in S$.
The set $\mathcal{L}_k(S)$ consists of all elements of $\mathcal{L}(S)$ of length $k$.
If $x\in\Omega_n$ then we define $\mathcal{L}(x)=\mathcal{L}(\{x\})$.

Given a nonempty set $\mathcal{W}$ of words we can define a set $X_\mathcal{W}\subset\Omega_n$
as a set of all $\omega\in\Omega$ such that $\mathcal{L}(\omega)\subset \mathcal{W}$.
It is well known (see \cite[Proposition 1.3.4]{LM}) that if $\mathcal{W}$ is a nonempty
collection of words such that
for every word $w\in\mathcal{W}$
all words appearing in $w$ are also
in $\mathcal{W}$ and at least one word among $w\alpha$, where $\alpha\in\Lambda_n$ is in $\mathcal{W}$,
then $X_\mathcal{W}$ is
an one-sided subshift and $\mathcal{L}(X_\mathcal{W})= \mathcal{W}$.

Let $P$ be a subset of positive integers. We say that a binary word $w=w_1\ldots w_l$ is \emph{$P$-admissible}
if $w_i=w_j=1$ implies $|i-j|\in P\cup\{0\}$. 
Let $\mathcal{W}(P)$ be the collection of all $P$-admissible words. By the result mentioned above,
$\Omega_P=X_{\mathcal{W}(P)}\subset\Omega_2$ is a binary subshift, and its language,
$\mathcal{L}(\Omega_P)$ is the set of all $P$-admissible words.
We will write $\sigma_P$ for $\sigma\colon \Omega_2\mapsto\Omega_2$ restricted to $\Omega_P$,
and call the dynamical system given by
$\sigma_P\colon\Omega_P\mapsto\Omega_P$ a \emph{spacing} shift given by $P$. If $w\in\lang(\Omega_2)$, then
by $[w]_P$ we denote $[w]\cap \Omega_P$.

It is easy to see that definition of a spacing shift implies that $N([1]_P,[1]_P)=P$.
Moreover, $\sigma_P$ is weakly mixing if and only if $P$ is a \emph{thick} set (see \cite{spacing,LZ,deVries}).

As we are concerned here with the entropy of subshifts of $\Omega_n$,
we recall here a definition of topological entropy suitable for our purposes.
If $X\subset \Omega_n$ is a subshift,
then we set $\lambda_k=\#\lang_k(X)$.
It is straightforward to see
that $\lambda_{m+n}\le \lambda_n\cdot\lambda_m$,
therefore the number
\[
h(X)=\lim_{k\to\infty}
\frac{\log\lambda_k}{k},
\]
is well defined, and actually  $h(X)=\inf \log\lambda_k / k$.
(Here, as elsewhere, we use logarithms with base $2$).
It is well known (see \cite{LM,W}) that
$h(X)$ is equal to the \emph{topological entropy}
of the dynamical system $(X,\sigma|_X)$.


\section{Hereditary subshifts and their topological entropy}

The aim of the present section is to provide a characterization of
hereditary subshifts with positive topological entropy. It will
allow us to describe topological and ergodic properties of the
hereditary subshifts with zero entropy.
Some of the results we include in this section are known
and can be proved using ergodic theory. Here we present them with new,
more elementary and straightforward proofs  which use only basic
combinatorics and topological dynamics to keep the exposition as self-contained as possible.
Nevertheless, we admit that the ergodic theory approach is undeniably elegant.

Recall, that a subshift $X\subset \Omega_n$ is \emph{hereditary} provided for any $\omega\in X$ if for some $\omega'\in\Omega_n$ we have $\omega'\le \omega$ (coordinate-wise), that is, $\omega'_i\le \omega_i$ for all $i\in\mathbb{N}$, then $\omega'\in X$. 
The following lemma follows directly from the definition of a hereditary subshift,
and records basic properties of hereditary subshifts for further reference.
Here, for a binary word $w=w_1\ldots w_k\in\lang_k(\Omega_2)$ we define
\[
\sum w = \sum_{i=1}^k w_i.
\]

\begin{lemma}\label{lemma:basic}
If $X\subset \Omega_n$ is a hereditary subshift, then
\begin{enumerate}
\item $0^\infty\in X$,
\item the atomic measure concentrated on $0^\infty$ is an invariant measure for $X$,
\item if $w=w_1\ldots w_k\in\lang(X)$, then 
\[2^{\#\{1\le i \le k: w_i\neq 0 \}} \le \# \lang_k(X).\]\label{cond:basic2}
\item there exists $\omega\in X$ such that the set
$\ones (\omega)=\{n\in\mathbb{N}:\omega_n=1\}$ have positive upper Banach (upper, asymptotic, respectively) density if and only if
there exists $\omega\in X$ such that the set
$\{n\in\mathbb{N}:\omega_n\neq 0\}$ have positive upper Banach (upper, asymptotic, respectively) density. \label{cond:basic4}
\end{enumerate}
\end{lemma}

Next result shows that the existence of a point with positive upper Banach density of the
occurrences of $1$'s is sufficient for a hereditary shift to have positive topological entropy.

\begin{lemma}\label{lem:positive-ubd-positive-entropy}
If $X\subset \Omega_n$ is a hereditary subshift and there exists $\omega\in X$ such that the set
$\ones (\omega)=\{n\in\mathbb{N}:\omega_n=1\}$ have positive upper Banach density, then $h(X)>0$.
\end{lemma}
\begin{proof}
By our assumption we can find $\eps>0$ and a sequence $w^{(k)}$ of words appearing in $\omega$ such that
$l(k)=|w^{(k)}|\to\infty$ with $k\to\infty$, and $\sum w^{(k)}\ge l(k)\eps$. By Lemma \ref{lemma:basic}\eqref{cond:basic2} we
have $l(k)\eps \le \log \# \lang_{l(k)}(X)$ for all $k>0$. It follows that
\[
h(X) = \lim_{n\to \infty}\frac{\log \# \lang_{n}(X)}{n}=\lim_{k\to \infty}\frac{\log \# \lang_{l(k)}(X)}{l(k)}\ge\eps,
\]
which concludes the proof.
\end{proof}

We will need the following simple combinatorial result whose proof can be found for example in \cite[p. 52]{Shields}.

\begin{lemma}\label{lemma:combinatorial}
Let $0<\eps\le 1/2$ and $n\ge 1$. Then
\[
\sum_{j=0}^{\lfloor n \eps \rfloor}{n \choose j}\le 2^{n\cdot H(\eps)},
\]
where $H(\eps)=-\eps\log \eps - (1-\eps)\log (1-\eps)$.
\end{lemma}

Let $X\subset\Omega_n$ be a subshift of the full shift over $\Lambda_n$.
For a symbol $\alpha\in\Lambda_n$ we define $\dind_k(X,\alpha)$ as the
maximal number of occurrences of the symbol $\alpha$ in a word $w\in\lang_k(X)$,
that is,
\[
\dind_k(X,\alpha)= \max\left\{\#\{1\le j \le k:w_i=\alpha\}\,:\,w\in\lang_k(X)\right\}.
\]
Clearly, $\dind_{s+t}(X,\alpha)\le \dind_s(X,\alpha)+\dind_t(X,\alpha)$ holds for any positive integers $s$ and $t$.
Therefore, the sequence $\dind_k(X,\alpha)$ is subadditive, and $\dind_k(X,\alpha)/k$ has a limit as $k$ approaches infinity.
Hence we can define \emph{maximal density of $\alpha$ in $X$} as
\[
\adsymb{\alpha}(X)=\lim_{n\to\infty}\frac{\dind_n(X,\alpha)}{n}=\inf_{n\ge 1}\frac{\dind_n(X,\alpha)}{n}.
\]
Theorem \ref{thm:density} is the best motivation for the above definition. Note that for a hereditary shift $X\subset\Omega_n$
we have
\[
\adsymb{n-1}(X)\le \adsymb{n-2}(X)\le\ldots\le \adsymb{1}(X)\le \adsymb{0}(X)=1.
\]

The following lemma follows from the ergodic theorem, but here we present a direct proof inspired by \cite{HLY}.

\begin{theorem}\label{thm:density}
If $X\subset\Omega_n$ is a subshift, then for every $\alpha\in\Lambda_n$ there exists a point $\omega\in X$
such that
\[
\ad(\{j:\omega_j=\alpha\})=\adsymb{\alpha}(X).
\]
\end{theorem}
\begin{proof}Without loss of generality we may assume that $n=2$ and $\alpha=1$.
If $\adone(X)=0$, then the set $\mathbb{N}\setminus\ones(\omega)$ must be thick for every $\omega\in X$.
Then $0^\infty\in X$ since $X$ is closed and shift invariant. We assume that $\adone(X)>0$.
For every $n>0$ let $\bar{w}^{(n)}=\bar{w}^{(n)}_1\ldots \bar{w}^{(n)}_n\in\lang_n(X)$ be a word of length $n$ such that
\[
\sum_{i=1}^n\bar{w}^{(n)}_i=\dind_n(X,1)= \max\left\{\sum_{i=1}^n w_i:w=w_1\ldots w_n\in\lang_n(X)\right\},
\]
and fix any point $\bar{x}^{(n)}\in[\bar{w}^{(n)}]_X$.

We claim that for each integer $k>0$ there exists a word $w^{(k)}\in\lang(X)$ such that
\begin{equation}\label{eq:star}
\adone(X)-\frac{1}{k}\le \frac{1}{j}\sum_{i=0}^{j-1} w^{(k)}_i \quad\text{for}\,1\le j \le k.
\end{equation}
For the proof of the claim, assume on contrary that \eqref{eq:star} do not hold for some $k>0$.
Then, $\adone(X)-1/k>0$. Set $m=k^2+1$. As we assumed that our claim fails,
for a point $y=\bar{x}^{(m)}$ defined above we can find a strictly increasing sequence
of integers $\{l(s)\}_{s=0}^\infty$ such that $l(0)=0$, $l({j})-l(j-1)\le k$, and
\[
\frac{1}{l({j})-l(j-1)}\sum_{i=l(j-1)}^{l({j})-1} y_i < \adone(X)-\frac{1}{k},
\]
for every $j=1,2,\ldots$. Let $t>0$ be such that $ l(t) \le m < l(t+1)$. Then
\begin{align*}
m\adone(X)  &\le \dind_m(X) = \sum_{i=1}^m \bar{w}^{(m)}_i=\sum_{j=0}^{t}\sum_{i=l(j-1)}^{l({j})-1}(l({j})-l(j-1)) y_i + \sum_{i=l(t)}^{m}y_i\\
            & < m(\adone(X)-\frac{1}{k})+k,
\end{align*}
contradicting the definition of $m$. Therefore, our claim holds.

Now, for each integer $k>0$ there exists a point $x^{(k)} \in[w^{(k)}]_X$, and since $X$ is compact, we may without loss
of generality assume that
$x^{(k)}$ converge to some $x\in X$. Hence for every $k>0$ there exists $N\ge k$ such that
$x|_{[0,k)}=w^{(N)}|_{[0,k)}$.  For every $k>0$ we have
\[
\adone(X)-\frac{1}{N}\le \frac{1}{k}\sum_{i=0}^{k-1} x_i \le \frac{\dind_k(X,1)}{k},
\]
where the first inequality follows by our claim, and the second is a consequence of the definition of $\dind_k(X,1)$.
We conclude the proof by passing to the limit as $k\to \infty$.
\end{proof}


It is clear that if there exists $\omega\in X$ such that $\ones (\omega)$ have positive upper Banach density, then $\adone(X)$
is also positive. Let us note an immediate consequence:

\begin{corollary}\label{cor:ad-ones}
If $X\subset\Omega_n$ is a subshift and $\ubd(\ones(x))>0$ for some $x\in X$, then there exits $y\in X$ such that $\ad(\ones(y))>0$.
\end{corollary}

We can now use the previous theorem and its corollary to provide a proof of \cite[Corollary to thm. 3.20]{F}.
\begin{theorem}\label{thm:F}
If $A\subset\mathbb{N}$ is a set of positive upper Banach density, then
there is a set $B\subset\mathbb{N}$ with positive density such that $B-B$ is contained in $A-A$.
\end{theorem}
\begin{proof} Let $P=A-A$. Then the characteristic function of $A$ denoted by $\omega_A$ belongs to the spacing shift $\Omega_P$.
By the Corollary~\ref{cor:ad-ones} there is a point $\omega\in \Omega_P$ with $\ad(\ones(\omega))>0$. Let $B\subset\mathbb{N}$ be such that
$\omega$ is its characteristic function. Then $\ad(B)>0$ and $B-B\subset P=A-A$.
\end{proof}

Let us note here yet another application of spacing shifts to combinatorial number theory. It follows directly from
Theorem~\ref{thm:minimality}.

\begin{lemma}
If $Z\subset\mathbb{N}$ is a piecewise syndetic set, then there is a syndetic set $S\subset\mathbb{N}$ such that $S-S\subset Z-Z$.
\end{lemma}

In the case of a binary subshift, we prove that $\adone(X)>0$ is necessary for $h(X)>0$.
\begin{theorem}\label{thm:zero-density-zero-entropy}
Let $X\subset \Omega_n$ be a subshift. If the maximal density of $\alpha$ in $X$ is zero ($\adsymb{\alpha}(X)=0$) for every $\alpha\in\Lambda_n\setminus\{0\}$, then $h(X)=0$.
\end{theorem}
\begin{proof}Without loss of generality we may assume that $n=2$ and $\alpha=1$. Fix $0<\eps<1/2$. As
\[
0=\adone(X)=\lim_{n\to\infty}\frac{\dind_n(X,1)}{n},
\]
there exists an $N=N(\eps)>0$ such that for each $n\ge N$ we have
\[
\dind_n(X,1)=\max\left\{\sum_{i=1}^n w_i:w=w_1\ldots w_n\in\lang_n(X)\right\}\le \lfloor n\eps\rfloor.
\]
It implies that
\[
\#\lang_n (X) \le \sum_{j=0}^{\lfloor n \eps \rfloor}{n \choose j} \qquad\text{for every }n\ge N.
\]
By Lemma \ref{lemma:combinatorial}, we get
\[
h(X)=\inf\frac{\#\log\lang_n (X)}{n}\le H(\eps),
\]
where $H(\eps)=-\eps\log \eps - (1-\eps)\log (1-\eps)$.
To finish the proof observe that $H(\eps)\to 0$ as $\eps\to 0$.
\end{proof}

Clearly, Theorems \ref{thm:density} and \ref{thm:zero-density-zero-entropy} imply:

\begin{corollary}
If $X\subset\Omega_n$ is a subshift and $h(X)>0$, then there is $\omega\in X$ such that 
$\ad(\{n:\omega_n=\alpha\})>0$ for some $\alpha\in\Lambda_n\setminus\{0\}$.
\end{corollary}

Finally, we state our main theorem characterizing hereditary shifts with positive entropy as the ones with positive
density of occurrences of $1$'s.

\begin{theorem}\label{thm:hereditary-entropy-density}
The topological entropy of a hereditary subshift $X\subset \Omega_n$ is positive
if and only if there exists $\omega\in X$ with $BD^*(\ones(\omega))>0$.
\end{theorem}
\begin{proof}Necessity of positive
density of occurrences of $1$'s follows from Lemma~\ref{lemma:basic}\eqref{cond:basic4}
and Theorem~\ref{thm:zero-density-zero-entropy},
sufficiency follows from Lemma~\ref{lem:positive-ubd-positive-entropy}.
\end{proof}

As remarked above we might take a different route and obtain an ergodic proof of
Theorem~\ref{thm:hereditary-entropy-density}.
It would hinge upon the Variational Principle for the topological entropy
and the well known result (see \cite[Lemma 3.17]{F}), which is included in
the first part of the following theorem (the equivalence of conditions \ref{one}-\ref{three}).
The other implications follows from Theorems \ref{thm:density} and \ref{thm:hereditary-entropy-density}.

\begin{theorem}\label{thm:ergodic-equivalences}
For a subshift $X\subset\Omega_n$ the following conditions are equivalent:
\begin{enumerate}
\item \label{one}There exists a point $\omega\in X$ such that $\ubd(\{n:\omega_n=\alpha\})>0$ for some $\alpha\in\Lambda_n\setminus\{0\}$.
\item \label{two}There exists a shift invariant measure $\mu$ on $X$ such that $\mu([\alpha]_X)>0$ for some $\alpha\in\Lambda_n\setminus\{0\}$.
\item \label{three}There exists a shift invariant ergodic measure $\mu_e$ on $X$ such that $\mu_e([\alpha]_X)>0$ for some $\alpha\in\Lambda_n\setminus\{0\}$.
\item \label{four}There exists a point $\omega\in X$ such that $\ad(\{n:\omega_n=\alpha\})$ exists and is positive for some $\alpha\in\Lambda_n\setminus\{0\}$.
\end{enumerate}
Moreover, if $X$ has positive topological entropy, then all the above conditions \ref{one}-\ref{four} must hold,
and if $X$ is a hereditary shift, then conditions \ref{one}-\ref{four} and $h(X)>0$ are equivalent.
\end{theorem}

We find it useful to slightly rephrase the previous theorem.

\begin{theorem}\label{thm:ergodic-equivalences2}
Let $X\subset \Omega_n$ be a subshift.
The following conditions are equivalent:
\begin{enumerate}
\item \label{one1}      For every $\alpha\in\Lambda_n\setminus\{0\}$ the cylinder $[\alpha]_X$ is universally null, that is,
                        $\mu([\alpha]_X)=0$ for any shift invariant measure on $X$.
\item \label{two1}      For every $\alpha\in\Lambda_n\setminus\{0\}$ and for every $\omega\in X$ we have $\ad(\{n:\omega_n=\alpha\})=0$.
\item \label{three1}    The atomic measure concentrated on $0^\infty$ is
                        the unique invariant measure for $X$.
\end{enumerate}
If any of the above conditions hold, then
\begin{itemize}
\item[$(\star)$] \label{five1} $h(X)=0$.
\item[$(\star\star)$] \label{four1} $X$ is proximal.
\end{itemize}
Moreover, if $X$ is hereditary, then all the above conditions \eqref{one1}-\eqref{three1}, and $(\star)$ 
are equivalent.
\end{theorem}

\begin{proof}
The equivalence of \eqref{one1}-\eqref{three1} follows from Theorem~\ref{thm:ergodic-equivalences}.
To see the condition \eqref{three1} implies the condition $(\star\star)$ we need two facts.
First says that a dynamical system is proximal
if and only if there is a fixed point $p\in X$ which is the unique
minimal point of the map $f$ (for a proof, see \cite[Proposition 2.2]{AK}).
The second is a well-known observation:
every minimal subsystem carries at least one invariant measure.
To finish the proof we invoke Theorem~\ref{thm:zero-density-zero-entropy}
and Theorem~\ref{thm:ergodic-equivalences}.
\end{proof}

Note that, even for hereditary shifts, the condition $(\star\star)$ above does not imply unique ergodicity, nor zero entropy,
which we will show later in Theorem~\ref{thm:kriz}.

Now we restrict ourselves back to the spacing shifts,
and turn our attention to the natural question: is there any property of $P$ that ensures $h(\Omega_P)>0$?
We have no satisfactory answer, but we will do show that this question is equivalent to the notoriously
elusive problem of characterization of the sets of (Poincar\'{e}) recurrence.

First, recall that a refinement of the classical Poincar\'{e} recurrence theorem motivates the following definition.
\begin{definition}
We say that $R\subset \mathbb{N}$ is the a \emph{set of recurrence} if for any measure preserving system $(X,\mathcal{X},\mu,T)$, and any set
$A\in\mathcal{X}$ with $\mu(A)>0$ we have $\mu(A\cap T^{-n}(A))>0$ for some $n\in R$.
\end{definition}
The following lemma is implicit in Furstenberg \cite[pp. 72-5]{F} (see also \cite{Bergelson}).
\begin{lemma}
A necessary and sufficient condition for $R\subset \mathbb{N}$ to be a set of recurrence is that for
every $A\subset \mathbb{N}$ with $\ubd(A)>0$ we have $(A-A)\cap R \neq\emptyset$.
\end{lemma}

By the above lemma we obtain the combinatorial characterization of sets of recurrence in terms of topological entropy of spacing shifts.
\begin{theorem}\label{thm:Poincare}
A set $R\subset \mathbb{N}$ is a set of Poincar\'{e} recurrence if and only if $h(\Omega_{\mathbb{N}\setminus R})=0$.
\end{theorem}
Recall that in \cite{spacing} the following problem is formulated (note that we slightly rephrased it below):
\begin{description}
\item[Question 5] Is there $P$ such that $\mathbb{N}\setminus P$ does not contain any $\text{IP}$-set but $\Omega_P$ is proximal?
Is there $P$ such that $\mathbb{N}\setminus P$ does not contain any $\text{IP}$-set but $h(\Omega_P)>0$?
Are these two properties (i.e. proximality and zero entropy) essentially different in the context of spacing subshifts?
\end{description}
To answer this question we will need the following lemma (see also \cite[Proposition 2.3]{YZ}).
\begin{lemma}\label{lem:pubd-implies-delta}
If $A\subset \mathbb{N}$ has positive upper Banach density, then there exists $k\in \mathbb{N}$ such that for every set
$B\subset\mathbb{N}$ with at least $k$ elements the difference set $A-A$ contains a positive element of $B-B$.
\end{lemma}
\begin{proof}By our assumption there is a positive number $\beta$ and a
sequence of intervals $[s_n,t_n]$ with $s_n,t_n\in\mathbb{N}$ and $t_n-s_n\to\infty$ as
$n\to\infty$ such that
\[
\lim_{n\to\infty}\frac{\#A\cap[s_n,t_n]}{t_n-s_n+1}=\beta>0.
\]
Let $k\in\mathbb{N}$ be such that $\beta>1/k$, and take any $B=\{b_1<b_2<\ldots<b_k\}$.

We will show that the sets $A_j=A+b_j$ for $j=1,\ldots,k$ can not be
pairwise disjoint. Assume on contrary that this is not the case. Let
$l_n=t_n-s_n+1$. Let $n$ be large enough to assure the following
\[
\frac{\#A\cap[s_n,t_n]}{t_n-s_n+1}
>\frac{1}{k}+\frac{b_k}{t_n-s_n+1}\qquad\text{and}\qquad t_n-s_n>b_k.
\]
Let
\[
C=\bigcup_{j=1}^k (A+b_j)\cap [s_n+b_j,t_n+b_j].
\]
Then $C \subset [s_n,t_n+b_k]$. Moreover,
for each $j$ the set $(A+b_j)\cap [s_n+b_j,t_n+b_j]$ has at least
$\lceil(t_n-s_n+1)/k\rceil+b_k$ elements.
Now the assumption that the sets $A_j=A+b_j$ for $j=1,\ldots,k$ are pairwise disjoint leads
to the conclusion that $C$ has more than $t_n-s_n+1+kb_k$ elements,
which gives us a contradiction.

Therefore $A_i\cap A_j\neq\emptyset$ for some $1\le i < j \le k$, hence
there are $a_i,a_j$ in $A$ and $b_i,b_j$ in $B$ such that $a_i-a_j=b_j-b_i$, which
concludes the proof.
\end{proof}

The following theorem generalizes \cite[Theorem 3.6]{spacing} since 
every $\text{IP}$-set is a $\Delta$-set.

\begin{theorem}
If the entropy of $\Omega_P$ is positive, then  $P$ intersects the difference set
of any infinite subset of integers, that is, $P$ is a $\Delta^*$-set.
\end{theorem}
\begin{proof}It is an immediate consequence of Theorem~\ref{thm:hereditary-entropy-density} and Lemma~\ref{lem:pubd-implies-delta}.
\end{proof}

Now, take the set $B=\{2^n-2^m:n>m\ge 0\}=\{2^k:k\ge 0\}-\{2^k:k\ge 0\}$, which is clearly a $\Delta$-set
To prove that $B$ is  not an $\text{IP}$-set, consider the binary expansions of elements of $B$, and observe that
each must be of the form
\[
\underbrace{1\ldots 1}_{a \text{ ones}}\underbrace{0\ldots 0}_{b \text{ zeros}},\text{ where }a>0,\,b\ge 0.
\]
Therefore there is no infinite set $A\subset B$ with $\text{FS}(A)\subset B$.
Hence the complement of $B$ in $\mathbb{N}$ is an $\text{IP}^*$-set which is not $\Delta^*$-set,
and we get the following corollary, which answers \cite[Question 5]{spacing}.
\begin{corollary}
There is a proximal spacing shift $\Omega_P$ with $P$ being an $\text{IP}^*$-set and
$h(\Omega_P)=0$.
\end{corollary}

It follows from Theorem~\ref{thm:ergodic-equivalences2}  that for a spacing shift zero entropy implies proximality, and
it will follow from Theorem~\ref{thm:kriz} that the converse is not true.


\section{Distributional chaos of hereditary shifts}\label{sec:dc}

In this section we consider \emph{distributional chaos} for hereditary shifts, generalizing
and extending results from \cite{spacing}.

Let $(X,f)$ be a dynamical system. Given $x,y\in X$ we define an \emph{upper} and \emph{lower distribution} function
on the real line by setting
\begin{align*}
  F_{xy}(t) &= \liminf_{n\to\infty}\frac{1}{n} \left\{0\le j \le n-1 : d(f^j(x),f^j(y))<t\right\},\\
F^*_{xy}(t) &= \limsup_{n\to\infty}\frac{1}{n} \left\{0\le j \le n-1 : d(f^j(x),f^j(y))<t\right\}.
\end{align*}
Clearly, $F_{xy}$ and $F^*_{xy}$ are nondecreasing, and $0\le F_{xy}(t)\le F^*_{xy}(t)\le 1$ for all real $t$.
Moreover, $F_{xy}(t)= F^*_{xy}(t)=0$ for all $t\le 0$, and $F_{xy}(t)= F^*_{xy}(t)=1$ for all $t>\diam X$.
We adopt the convention that $F_{xy} < F^*_{xy}$
means that $F_{xy}(t) < F^*_{xy}(t)$ for all $t$ in some interval of
positive length.

Following \cite{BSS} we say that a pair $(x,y)$ of points from $X$ is a DC$1$-\emph{scrambled  pair} if
$F^*_{xy}(t)=1$ for all $t>0$, and $F_{xy}(s)=0$ for some $s>0$.
A pair $(x,y)$ is a DC$2$-\emph{scrambled  pair} if
$F^*_{xy}(t)=1$ for all $t>0$, and $F_{xy}(s)<1$ for some $s>0$.
Finally, by a DC$3$-\emph{scrambled pair} we mean a pair $(x,y)$ such that
$F_{xy} < F^*_{xy}$.
The dynamical system $(X,f)$ is distributionally chaotic of type \emph{i}
(or DC\emph{i}-chaotic for short) where $i=1,2,3$, if there is an uncountable
set $S\subset X$ such that any pair of distinct points from $S$ is DC\emph{i} scrambled.

The proof of the following lemma is a standard exercise, therefore we skip it.

\begin{lemma} \label{lem:densities-and-distances}
Let $X\subset\Omega_n$ be a subshift, and let $x,y\in X$. Then
\begin{enumerate}
\item The following conditions are equivalent:
\begin{enumerate}
\item \label{dd1a}  $F_{xy}(s)<1$ for some $s\in (0,\diam X]$,
\item \label{dd1b}  for any $k\ge 0$ the set $\{n\in\mathbb{N}:x_{[n,n+k]}\neq y_{[n,n+k]}\}$
has positive upper density,
\item \label{dd1c}  the set $\Diff(x,y)=\{n\in\mathbb{N}:x_n\neq y_n\}$
has positive upper density.
\end{enumerate}
\item The following conditions are equivalent:
\begin{enumerate}
\item \label{dd2a} $F^*_{xy}(t)=1$ for all $t\in (0,\diam X]$,
\item \label{dd2b} for any $k\ge 0$ the set
$\{n\in\mathbb{N}:x_{[n,n+k]}= y_{[n,n+k]}\}$
has upper density one,
\item \label{dd2c} the set $\Equal(x,y)=\{n\in\mathbb{N}:x_{n}= y_{n}\}$
has upper density one.
\end{enumerate}
\end{enumerate}
\end{lemma}

\begin{lemma}\label{lem:set-theory}
For every set $S\subset\mathbb{N}$ with $\ud(S)>0$ there is $S_0\subset S$
such that
\begin{equation}\tag{$\ast$}\label{cond:a}
\ud(\{n\in\mathbb{N}:\{n,n+1,\ldots,n+k-1\}\subset \mathbb{N}\setminus S_0\})=1 \qquad\text{for each $k\in\mathbb{N}$},
\end{equation}
and an uncountable family $\Gamma$
of subsets of $S_0$ such that for every $S',S''\in\Gamma$, $S'\neq S''$ we have
\[
\ud(S'\setminus S'')=\ud(S''\setminus S')=\ud(S).
\]
\end{lemma}
\begin{proof}Let $\alpha=\ud(S)>0$.
There exists an increasing sequence of positive integers $b_1<b_2<\ldots$ such that
\[
\lim_{n\to\infty} \frac{1}{b_n}\#\{1\le j \le b_n: j\in S\}=\alpha.
\]
Without loss of generality we may assume that $n\cdot b_n\le b_{n+1}$ for all $n\in\mathbb{N}$.
For $n\in\mathbb{N}$ let
\[
S_n=(b_{2n-1},b_{2n}]\cap S \qquad\text{and}\qquad S_0=\bigcup_{n=1}^\infty S_n.
\]
Since $(b_{2n},b_{2n+1}]\subset \mathbb{N}\setminus S_0$ for each $n$ we have
\[
\frac{1}{b_{2n+1}}\#\{1\le j \le b_{2n+1}: j\notin S_0\}\ge \frac{b_{2n+1}-b_{2n}}{b_{2n+1}}\ge 1-\frac{1}{2n},
\]
and therefore \eqref{cond:a} holds. Note that
\[
\frac{1}{b_{2n}}\#\{1\le j \le b_{2n}: j\in S_0\}\ge \frac{\#\{1\le j \le b_{2n}: j\in S\}}{b_{2n}}\-\frac{b_{2n-1}}{b_{2n}},
\]
hence if $A$ is an infinite set of positive integers then
\[
\ud(S(A))=\ud(S),\qquad\text{where }S(A)=\bigcup_{n\in A} S_{n}.
\]
To finish the proof it is enough to observe that there exists an uncountable family $\Theta$ of
infinite sets of positive integers such that for any $A,B\in\Theta$ with $A\neq B$ the sets
$A\setminus B$ and $B\setminus A$ are infinite.
\end{proof}

\begin{lemma}\label{lem:dc-construction}
Let $X\subset\Omega_n$ be a hereditary subshift.
If  $x$ and $y$ is a pair of points in $X$ such that $F_{xy}(s)<1$ for some $s>0$, then
there exists an uncountable set $\Gamma\subset X$ such that for every $u,v\in\Gamma$, $u\neq v$ we have
\begin{enumerate}
  \item $F^*_{uv}(t)=1$ for all $t\in (0,\diam X]$,
  \item $F_{uv}(s)=F_{xy}(s)<1$.
\end{enumerate}
In particular, any pair $(u,v)$ with $u\neq v$ is DC$2$-scrambled, (DC$1$-scrambled, if in addition we have $F_{xy}(s)=0$).
\end{lemma}
\begin{proof}
Let $(x,y)$ be a pair of points such that $F_{xy}(t)<1$ for some $t>0$.
By Lemma \ref{lem:densities-and-distances}(1) we get that $\ud(\{n:x_n\neq y_n\})>0$.
Since $X$ is hereditary without loss
of generality we may assume that $\ud(\ones(x))>0$. With the customary abuse of notation,
we let $\Gamma$ to be the set of characteristic functions of subsets of $S=\ones(x)$
provided by Lemma~\ref{lem:set-theory}. Now, we apply both parts of Lemma~\ref{lem:densities-and-distances}
to see that each pair of different points of $\Gamma$ fulfills the desired conditions.
\end{proof}

\begin{theorem}\label{thm:dc2}
Let $X \subset\Omega_n$ be a hereditary subshift. Then the following conditions are equivalent
\begin{enumerate}
\item \label{dca} The topological entropy of $X$ is positive.
\item \label{dcb} There exists points $x,y\in X$ such that $F_{xy}(t)<1$ for some $t>0$.
\item \label{dcc} $X$ is DC$3$-chaotic.
\item \label{dcd} $X$ is DC$2$-chaotic.
\end{enumerate}
\end{theorem}

\begin{proof} On account of Lemma \ref{lem:dc-construction} conditions (\ref{dcb}-\ref{dcd}) are equivalent.
By Theorem~\ref{thm:ergodic-equivalences} positive entropy of $X$ is equivalent to the existence of a point
$x\in X$ with $\ad(\ones(x))>0$. Now we may consider a pair $(x,y)$ where $y=0^\infty$,
and apply Lemma \ref{lem:dc-construction}
to finish the proof.
\end{proof}
Note that the implications $\eqref{dca}\implies\eqref{dcd}\implies\eqref{dcc}\implies\eqref{dcb}$ of the theorem above
also hold for general dynamical systems, and are trivial, except
$\eqref{dca}\implies\eqref{dcd}$, which had been a longstanding open problem
solved recently by Downarowicz in \cite{D}.
The only implication specific for hereditary shifts is $\eqref{dcb}\implies\eqref{dca}$.

\begin{theorem}\label{thm:dc1}
A hereditary shift $X\subset\Omega_n$ is DC$1$-chaotic  if and only if $X$ is not proximal.
\end{theorem}
\begin{proof}
If $\omega=(\omega_i)\neq 0^\infty$ is a minimal point, then $x=\sigma^\nu(\omega)\in [\alpha]$ for some $\nu\ge 0$
and $\alpha\in\Lambda_n\setminus\{0\}$. 
Moreover, $x$ is also a minimal point of $X$, and hence it returns to the cylinder $[\alpha]$ syndetically often,
that is, there is $k>0$ such that $x_{[j,j+k)}\neq 0^k$ for each $j\in\mathbb{N}$. Let $y=0^\infty$. Therefore
$(x,y)$ is a pair such that $F_{xy}(2^{-k})=0$. We conclude from Lemma \ref{lem:dc-construction} that there
must be an uncountable DC$1$-chaotic set in $X$. For the other direction, note that
by \cite[Corollary 15]{O} the DC$1$-scrambled pairs are absent in any proximal system. Hence, DC$1$-chaos implies existence
of a minimal set other than $0^\infty$.
\end{proof}

The following theorem completes our answer to \cite[Questions 4 and 5]{spacing}. Note that such a subshift
we obtain by this theorem has an invariant measure supported outside minimal sets. The first example of
this phenomenon was given by Goodwyn in \cite{G}. Here, following \cite{McC}
by an \emph{$r$-coloring} of $\mathbb{N}$ we mean any partition $\mathbb{N}=C_1\cup\ldots\cup C_r$. 
The indices $1,2,\ldots,r$ are called the \emph{colors}. A set $E\subset\mathbb{N}$ is said to 
be \emph{$r$-intersective} if for every $r$-coloring $\mathbb{N}=C_1\cup\ldots\cup C_r$ there 
exists a color $i$ such that $(C_i-C_i)\cap E$ is non-empty. We say that $E\subset \mathbb{N}$ 
is \emph{chromatically intersective} if $E$ is $r$-intersective for any $r\ge 1$. 
By \cite[Proposition 0.12]{McC} a set $E$ is chromatically intersective if and only if whenever
$(X,f)$ is a dynamical system, $x\in X$ is a minimal point and $U\subset X$ is an open neighborhood 
of $x$ then there is $n\in E$ such that
$U\cap f^{-n}(U)$ is non-empty.

\begin{theorem}\label{thm:kriz}
There exits a weakly mixing and proximal spacing shift $(\Omega_P,\sigma_P)$ with positive topological entropy.
Hence, there is a DC$2$-chaotic spacing shift which is not DC$1$-chaotic.
\end{theorem}
\begin{proof}By the result of K\v{r}\'{\i}\v{z} (proved first by \cite{Kriz}, here we use \cite[Theorem 1.2]{McC}) 
there exists a set $A\subset\mathbb{N}$
with $\ud(A)>0$ such that $(A-A)\cap C=\emptyset$ for some
chromatically intersective set $C$. Let $P=\mathbb{N}\setminus C$.

We claim that the spacing shift $\Omega_P$ is proximal, that is, we claim that $0^\infty$
is the unique minimal point of $\Omega_P$. Assume on contrary that there is another minimal point $\omega\in\Omega_P$. 
Then there is some $k\ge 0$ such that $[1]_P$ is an open neighborhood of a minimal point $\sigma^k(\omega)$. 
By \cite[Proposition 0.12]{McC} there must be $n\in N([1]_P,[1]_P)\cap C$, but this contradicts the definition of $P=\mathbb{N}\setminus C$. 
So $\Omega_P$ is proximal. 

Moreover, the characteristic function of the set $A$ belongs to $\Omega_P$, hence $h(\Omega_P)>0$, since $\ud(A)>0$. 
By Theorems \ref{thm:dc2} and \ref{thm:dc1} the spacing shift $\Omega_P$ is a DC$2$-chaotic but it is not DC$1$-chaotic. 
To prove that $\Omega_P$ is weakly mixing we need to show that $C$ can be chosen so that $P=\mathbb{N}\setminus C$ is thick. 
Since most of the construction of the set $C$ can be repeated without introducing anything new, we ask the reader to re-examine the proof of \cite[Theorem 1.2]{McC} to see that $C$ is defined as an union of finite sets
\[
C=C_1\cup (m_1n_1)\cdot C_2 \cup (m_1n_1m_2n_2)\cdot C_3\cup\ldots,
\]
where $c\cdot J=\{cj:j\in J\}$, and positive integers $n_1,n_2,\ldots$ can be chosen to be arbitrarily large. As all sets $C_1,C_2,\ldots$ are finite, and
do not depend on $n_i$'s, one can force $C$ to have thick complement.
\end{proof}

\section{Beta shifts are hereditary}\label{sec:beta}

We prove here that the very important class of beta shifts provides 
a whole family of examples of hereditary shifts.
We follow the description of beta shifts presented in \cite{Th}.
To define a beta shift fix a real number $\beta>1$ and
let the sequence $\omega^{(\beta)}\in \Omega_{\lceil\beta\rceil}$
be the expansion of $1$ in base $\beta$, that is,
\[
1=\sum_{i=1}^\infty \omega^{(\beta)}_i\beta^{-i}.
\]
Then $\omega^{(\beta)}\in \Omega_{\lceil\beta\rceil}$
is given by $\omega^{(\beta)}_1=\lfloor\beta\rfloor$ and
\[
\omega^{(\beta)}_i=
\left\lfloor\beta^i\left(
1-\omega^{(\beta)}_1\beta^{-1}-\omega^{(\beta)}_2\beta^{-2}-\ldots-\omega^{(\beta)}_{i-1}\beta^{-i+1}
\right)\right\rfloor.
\]
Let $\preceq$ denote the lexicographic ordering of the set $(\mathbb{N}\cup\{0\})^\mathbb{N}$. Then
it can be proved that for any $k\ge 0$ we have
\begin{equation}\label{eq:beta}
\sigma^k(\omega^{(\beta)})\preceq \omega^{(\beta)},
\end{equation}
where $\sigma$ denotes the shift operator on $(\mathbb{N}\cup\{0\})^\mathbb{N}$. 
By a result of Parry \cite{Parry}, the converse is also true,
that is, if any sequence over a finite alphabet satisfies the above equation then there is
a $\beta>1$ such that this sequence is a $\beta$-expansion of $1$. It follows from
\eqref{eq:beta} that
\[
\Omega_\beta=\{\omega\in\Omega_{\lceil\beta\rceil}:\omega_{[k,\infty)}\preceq \omega^{(\beta)}\text{ for all }k\ge 0\}
\]
is a subshift of $\Omega_{\lceil\beta\rceil}$, called the beta subshift defined by $\beta$.

It is easy to see that the above description of beta shifts implies the following:
\begin{lemma}
Every beta shift $\Omega_\beta\subset\Omega_{\lceil\beta\rceil}$ is hereditary.
\end{lemma}

\section{Final remarks and an open problem}\label{sec:final}

Finally, we present an example, which shows that there exist hereditary shifts other than spacing shifts or beta shifts.

\begin{theorem}\label{thm:mix}
There exists mixing, hereditary binary subshift without any DC$3$-scrambled pair, which is not conjugated to any spacing shift, nor any beta shift.
\end{theorem}
\begin{proof} To specify $X$ we will describe the language of $X$. Let $\mathcal{W}$ be the collection of all $w$ words from $\lang(\Omega_2)$ such that
for any word $u$ occurring in $w$ if $2^{k-1}+1\le |u| \le 2^{k}$, then the symbol $1$ occurs at less than $k+1$ positions in $u$. It is clear that $\mathcal{W}$ fulfills the assumptions of \cite[Proposition 1.3.4]{LM}, and hence $X=X_\mathcal{W}$ is a binary subshift with $\mathcal{W}=\lang(X_{\mathcal{W}})$. Then clearly, $X$ is hereditary, and $\ad(\omega)=0$ for every $\omega\in X$, hence the topological entropy of $X$ is zero, and there is no DC$3$-scrambled pair in $X$. Now fix any two cylinders $[u]$ and $[v]$ in $X$. Since $u0^kv0^\infty\in X$ for all sufficiently large $k$, we conclude
by Theorem \ref{thm:transitivity} that $X$ is mixing.
It follows from  \cite{spacing} that all mixing spacing shifts have positive topological entropy.
On the other hand it is well known that the topological entropy of every beta shift is also positive.
Hence $X$ is not conjugated to any spacing shift nor beta shift.
\end{proof}

As the topological entropy of a beta shift $\Omega_\beta$ is $\log\beta$, using
the main result of \cite{D} or Theorem \ref{thm:dc2} we obtain that
every beta shift is DC$2$-chaotic. But actually more is true.

\begin{theorem}
Every beta shift is DC$1$ chaotic.
\end{theorem}
\begin{proof} It is well known that beta shifts are never proximal (it follows for example
from the main result of \cite{Sig} or \cite[Proposition 5.2]{Th}).
Then one invokes Theorem \ref{thm:dc1} to finish the proof.
\end{proof}

Note that it is known that all beta shifts have
the unique measure of maximal entropy (see \cite{Sig}). It prompts us to state
the following conjecture which to our best knowledge remains open.

\begin{description}
\item[Conjecture] Every hereditary shift is \emph{intristically ergodic}, that is, it posses
the unique measure of maximal entropy.
\end{description}


\end{document}